\documentclass[11pt]{amsproc}
\usepackage{amstext,amsmath,amssymb,amsthm, esint}
\usepackage{latexsym}
\usepackage{exscale}
\usepackage{color}

\RequirePackage[cp1251]{inputenc}
\RequirePackage{srcltx}

\numberwithin{equation}{section}
\RequirePackage{srcltx}

\RequirePackage[cp1251]{inputenc}

\numberwithin{equation}{section}

\newcommand\ov{\overline}

\renewcommand\r{\rangle}
\renewcommand\l{\langle}

\newcommand\dsize{\displaystyle}

\newcommand\cal{\mathcal}

\newcommand\R{\mathbb{R}}
\newcommand \C{\mathbb{C}}
\newcommand\Q{\mathbb{Q}}
\newcommand\Z{\mathbb{Z}}
\newcommand\N{\mathbb{N}}
\renewcommand\S{\mathbb{S}}
 \newcommand\A{{\bf A}}
 \newcommand\M{{  M}}

\newcommand\B{\mathcal{B}}
\newcommand\s{\mathcal{S}}

\newtheorem{Thm}{Theorem}[section]
\newtheorem{Lemma}[Thm]{Lemma}
\newtheorem{Cor}[Thm]{Corollary}

\theoremstyle{remark}

\begin{document}

\title[]{Exponential bases on  multi-rectangles of $\R^d$}
\author{Laura De Carli}
\address{L.~De Carli, Florida International University,
Department of Mathematics,
Miami, FL 33199, USA}
\email{decarlil@fiu.edu}
 \begin{abstract}
 
 We produce explicit exponential bases on finite    union of disjoint   rectangles in $\R^d$ with   rational vertices.   
 The proof of our main result relies on the semigroup properties and   precise   norm  estimates of a remarkable   family of  linear operators on $\ell^2(\Z^d)$ that generalize  the discrete Hilbert transform. 
\end{abstract} 
 
\maketitle

\section{Introduction}

The main purpose of this paper is to   construct explicit Riesz bases made of    exponential functions  (or: {\it exponential bases}) on  multi-rectangles  in $\R^d$ with rational vertices. By {\it multi-rectangle} we mean a finite    union of disjoint  rectangles in the form of   $\prod_{j=1}^d [a_j,\, b_j)$, with $-\infty <a_j<b_j <\infty$.  In  this paper we assume
$a_j$, $b_j\in\Q$.  We may denote with  {\it multi-interval}    a finite    union of disjoint  intervals in $\R$, i.e., a multi- rectangle in dimension $d=1$. We have stated  the definition of Riesz basis and other preliminary results in Section 2. 

Exponential bases are known to exist on any multi-interval of $\R$  and on any  multi-rectangle in $\R^d$. 
 See the recent \cite{KN} and  \cite{KN2}, and see also 
 \cite{Marzo} and \cite{DK}.   In these papers no explicit exponential bases are produced.

Since we consider only multi-rectangles with vertices in $\Q^d$, 
after perhaps  a  translation  and dilation of coordinates  (see Corollary \ref{Cor-rational-cubes}) we   can restrict our attention to   multi-rectangles    with   vertices in   $(\frac 12+\Z)^d$,  the lattice of points  in $\R^d$ with half-integer coordinates. 
Such multi-rectangles  can be written as  disjoint union  of translated of the unit cube   $Q_0= [-\frac 12, \ \frac 12)^d$. 
 We   let 
\begin{equation}\label{defQ}
Q= Q(\vec M_1, \, ...,\,\vec M_N)=\bigcup_{p=1}^N   Q_0+\vec M_p, 
\end{equation} 
where we assume $  \vec M_p\in\Z^d$ and $\vec M_p\ne  \vec M_q \mbox{ if }  p\ne q$.
 We  find exponential bases of $L^2(Q)$  in the form of 
\begin{equation}\label{d2-defofB}
  \B= \B(\vec\delta_1, \, ...,\, \vec \delta_N)= \bigcup_{j=1}^{N } \{e^{ 2\pi i \l \vec n+\vec \delta_j,\,  x\r}\}_{\vec n\in\Z^d} \end{equation}
  where  $\vec \delta_p=(\delta_{p,1}, ...,\, \delta_{p,d}) \in\R^d$.  Each set $\{e^{ 2\pi i \l \vec n+\vec \delta_j,\,  x\r}\}_{\vec n\in\Z^d}$ is an ortho-normal exponential basis on any  unit cube of    $Q$. 
  The idea of combining exponential bases on individual rectangles to form a basis on  their disjoint union is   not new. See  e.g. the introduction of \cite{KN}, and \cite{KN2}  and \cite[Sect. 4]{LS}.  Our main result is the following
 
  \begin{Thm}\label{T1-N-interv}
   $\B$   is a Riesz basis of  $L^2(Q)$  if and only if  the matrix
    \begin{equation}\label{defGamma} {\bf \Gamma}= {\bf \Gamma}(\l\vec M_p,\, \vec\delta_j\r)=\{ e^{2\pi i \l\vec \delta_j, \vec M_p\r}\}_{1\leq j,p\leq N} \end{equation} is nonsingular.  
   %
   The optimal frame constants   of $\B$  are  the maximum and minimum singular value of ${\bf \Gamma}$.
    
   \end{Thm}
 \medskip
 We recall that the singular values of a matrix ${  M}$ are the   eigenvalues of ${  M^*  M} $, where ${ M^*}$ is the   conjugate   transpose of ${  M}$.
 
The  optimal frame constants of $\B$ can  be explicitly evaluated  when  $N=2$ (see Corollary \ref{C4-2 cubes}); the case  $N>2$  is discussed in    Section \ref{S-est-eigen}.

 Exponential bases  or frames on multi-intervals of the real line   have been investigated by several authors. In addition to \cite{KN}, see  also  \cite{BK},  and \cite[Sect. 4]{LS} and the references cited in these papers. 
The case $d=1$ of   Theorem \ref{T1-N-interv} extends  the   main theorem   and Remark 4  in \cite{BK}.
  

\medskip
In general, it is not true  that an exponential frame  contains a Riesz basis  or that an exponential Riesz sequences can be completed to a   Riesz basis (see e.g. \cite{S2}). But when $Q=Q(\vec M_1,\, ...,\, \vec M_N)$ is as in \eqref{defQ},   sets  of exponentials  $\B=\B(\vec\delta_1, ...,\vec\delta_N)$ as in   \eqref{d2-defofB}  enjoy the following remarkable properties.

\begin{Thm} \label{Cor-equiv-bases-frames} 
   The following  are equivalent:
 \begin{itemize}
 \item $\B$ is a Riesz sequence   in   $L^2(Q)$
 \item $\B$ is a frame 
 \item $\B$ is a Riesz basis.
 \end{itemize}
 \end{Thm}

\begin{Thm}\label{C-meas=0}   The set 
$$
 {\cal L}= \{ (\vec\delta_1, ...,\vec\delta_N)\in (\R^d)^N\ :\  \B(\vec\delta_1, ...,\vec\delta_N) \ \mbox{   is not a Riesz basis of $L^2(Q)$}\} 
 $$ has Lebesgue measure zero. 
\end{Thm}

\medskip
 The following is   a special and significant case of Theorem \ref{T1-N-interv}.

 \begin{Thm}\label{T1-special-delta}
  Let $\vec\delta\in \R^d$ be fixed; the set  \begin{equation}\label{d2-def-S}
 \s(\vec \delta)=  \cup_{j=1}^{N } \{e^{ 2\pi i  \l \vec n +(j-1)\vec \delta, \vec x\r}\}_{\vec n\in\Z^d}.  
  \end{equation} 
 is a  Riesz basis of $L^2(Q)$ if and only if
 \begin{equation}\label{e2-cond=delta} \l \vec M_p - \vec M_q,\ \vec \delta\r\not\in\Z,
 \quad   \mbox{ for every $ 1\leq p\ne q \leq N$}.
 \end{equation} 
The  frame constants  of $\s(\vec \delta)$ are the maximum and  minimum eigenvalue of the matrix ${\bf\tilde  B}=\{\tilde \beta_{p,q}\}_{1\leq p,q\leq N}$, where 
\begin{equation}\label{def-betatilde-Thm13}
\tilde \beta_{p,q}=\begin{cases}  \frac{\sin(\pi N \l \vec\delta ,\, \vec M_q-\vec M_p\r)}{
\sin(\pi   \l \vec\delta ,\, \vec M_q-\vec M_p\r)} & \mbox{ if $p\ne q$},
\cr
N & \mbox{if $p=q$}
\end{cases}
\end{equation}
 \end{Thm}

\medskip
The proof of  Theorem \ref{T1-N-interv}  relies on the semigroup properties and   precise   norm  estimates of a remarkable   family of  linear operators on $\ell^2(\Z^d)$ that we   discuss in section \ref{S-isometry}. 

The rest of the paper is organized as follows. In Section \ref{S-proofs} we prove our main results. We have collected a number of corollaries  and examples in Section  \ref{S-Corollaries}. In Section \ref{S-est-eigen} we estimate the frame constants of $\B$ and $\s(\vec \delta)$.  

 \section{Preliminaries}\label{S-Prelim}
 \subsection{Bases and frames} We have used   \cite{Lang} for standard linear algebra results  and the excellent textbook  \cite{Heil}    for definitions and properties of  bases and frames in Hilbert spaces. See also \cite{Cr} and   \cite{Y}.   
  
 \medskip
 A sequence of vectors ${\cal V}= \{v_j\}_{j\in\N}$  in a separable Hilbert space  $H$ is a {\it frame} if 
there exist  constants $A, \ B>0$    such that for every $w\in H$,
\begin{equation}\label{e2-frame}
 A||w||^2\leq  \sum_{j=1}^\infty |\l  w, v_j\r |^2\leq B ||w||^2.
\end{equation}
Here, $\langle\ ,\ \rangle $  and $||\ ||=\sqrt{\l \ , \    \r} $    are the  inner product   and the norm  in $H$.   The sequence ${\cal V}$  is a   {\it  tight frame } if $A=B$,   it is a {\it Parseval frame} if $A=B=1$, 
and a {\it Riesz sequence}  if  
 the following
inequality  is satisfied  for all finite sequences $ \{a_j\}_{j\in J}\subset\C $.
\begin{equation}\label{e2- Riesz-sequence}
 A  \sum_{j\in J}   |a_j|^2   \leq  \left\Vert \sum_{j\in J}  a_j  v_j \right\Vert^2  \leq B \sum_{j\in J} |a_j|^2 .
\end{equation}
A {\it Riesz basis} is a frame and a Riesz sequence, i.e.,    a set  of vectors  that satisfies  \eqref{e2-frame} and \eqref{e2- Riesz-sequence}.

If $H=L^2(D)$, with  $D\subset \R^d$ of finite Lebesgue measure $|D|$,   Riesz bases (or frames) made   of exponential functions  are especially relevant in the applications.   An {\it exponential  basis} of $L^2(D)$  is a Riesz basis in the form of  $E(\Lambda)=\{ e^{2\pi i \langle   \vec\lambda ,\,  x\rangle}\}_{ \vec\lambda\in\Lambda}$, where $\Lambda$ is a discrete set of $\R^d $.
Exponential bases  are important 
to provide  unique and stable  representation of functions in $L^2(D)$  in terms of the functions in  $E(\Lambda)$, 
 with  coefficients that are easy to calculate. 
Unfortunately,  our understanding of exponential  bases is still very incomplete. 
  There are very few examples of domains  in which it is known how to construct   exponential bases,
 and no example of domain  for which 
 exponential bases are known  not to exist. See   \cite{GL}, \cite{K2}  and  the   references
cited there.  

Because  frames  are not necessarily linearly independent,  
they are often  more easily constructible than bases. 
For example, when $D\subset Q_0=[-\frac 12, \frac 12)^d$  we can easily verify that    $E(\Z^d) $  is a Parseval  frame on $D$.  See also \cite[Prop. 2.1]{DK}.
 The construction of exponential frames on unbounded sets of finite measure is   a difficult problem  that has  been recently solved in  \cite{NOU}. 
 
\subsection{Exponential bases on $L^2(Q_0)$}
It is proved in   \cite{LRW}   that        $E(\Lambda)$ is  an orthonormal basis on $L^2(Q_0)$ 
if and only if  the sets  $\{ Q_0+\vec\lambda \}_{\vec\lambda\in\Lambda} $ {\it tiles  $\R^d$}; that is, if 
  $\cup_{\vec\lambda\in\Lambda  }(Q_0+\vec\lambda)=\R^d $    and $|(Q_0+\vec\lambda_j)\cap (Q_0+\vec\lambda_i)|=0$  whenever $\lambda_i\ne \lambda_j$.

    A bounded domain $D\subset\R^d$ is called {\it spectral} if $L^2(D)$ is  has an orthogonal exponential basis. The connection between tiling and spectral properties of   domains of $\R^d$  is deep and fascinating and has spur intense investigation since when B. Fuglede formulated his famous tiling $\iff$ spectral conjecture in \cite{F}.  See  also \cite{K} and the references cited there.

Non-orthogonal exponential bases in  $L^2(-\frac 12,\frac 12)$ 
were first  investigated  by Paley and Wiener \cite{PW}  and Levinson\cite{Levinson} and extensively studied by several other authors.    A complete characterization of exponential bases on $L^2(-\frac 12,\frac 12)$ was given by 
Pavlov   in \cite{P}.  
 It is proved in \cite[Lemma 2.1]{SZ} that if $\Lambda=\{( \lambda_{n_1}, \, ...,\, \lambda_{n_d})\}_{(n_1, ... n_d)\in\Z^d}$,  the set
  $E(\Lambda)$ is an  exponential basis  on $Q_0=[-\frac 12, \frac 12)^d $ if and only if  the sets $\{e^{2\pi i  \lambda_{n_j} x_j}\}_{n_j\in\Z}$ are 
exponential  bases  on $ [-\frac 12, \frac 12)$. To   the best of our knowledge, no complete characterization of exponential bases on $L^2(Q_0 )$ exists in the literature.

\medskip
\subsection{Stability of Riesz bases.} Riesz bases are {\it stable}, in the sense that a small perturbation of a Riesz basis produces a Riesz basis. The celebrated Kadec stability theorem states that  any  set   $ \{ e^{2\pi i\lambda_n  x}\}_{n\in\Z}$ is Riesz basis of $L^2(-\frac 12,\,\frac 12)$ if    $ |\lambda_n-n|\leq L <\frac 14$ whenever $n\in\Z$.   In \cite{Levinson} it is shown   that the constant $\frac 14$ cannot be replaced by any larger constant. 
See \cite{Ka} or \cite{Y} for a proof of Kadec theorem. 

The following  multi-dimensional generalization of Kadec theorem is  in   \cite{SZ}. \begin{Thm}\label{T2-multi-Kad}
   Let $\Lambda= \{ \lambda_{\vec n}=(\lambda_{\vec n,1},\, ...,\, \lambda_{\vec n,d})\in\R^d\}_{\vec n\in\Z^d}$ be a sequence 
in $\R^d$ for which 
   $ \dsize || \vec\lambda_{\vec n }-  \vec n ||_\infty =\sup_{1\leq j\leq d} 
   |\lambda_{\vec n,j} -n_j|  \leq L<\frac 14$  
 whenever $\vec n\in \Z^d$. Then,  $E(\Lambda)$ is an exponential basis of $L^2(Q_0)$ and the constant $\frac 14$ cannot be replaced by any larger constant.
 \end{Thm}

\subsection{Three useful lemmas}

\begin{Lemma}\label{L-same-consts}
Let ${\cal V}$ be a Riesz basis in $H$;    
  the optimal  constants $A$ and $B$ in  the inequalities \eqref{e2-frame} and \eqref{e2- Riesz-sequence}  are the same. 
\end{Lemma}
\begin{proof}
Follows from   \cite[Proposition  3.5.5]{Cr}.
\end{proof}

 The following   is   Lemma 3 in  \cite {KN2}. 
\begin{Lemma} \label{lemma3}  $E(\Lambda)\subset E(\Z^d)$  is a frame on a domain $ D \subset Q_0$  if and only if $E(\Z^d -\Lambda)$ is a Riesz
sequence on $Q_0-D$.
\end{Lemma}

The following is  Proposition 3.2.8 in \cite{Cr}
\begin{Lemma}\label{L-normalized-tight}
A sequence of  unit vectors   ${\cal V }\subset H$ is a Parseval    frame   if and only if it  is an orthonormal Riesz basis.
\end{Lemma}

From Lemma \ref{L-normalized-tight} follows that an exponential basis of $L^2(D)$ is orthognal if and only if it is a tight frame of $L^2(D)$ with frame constant $|D|$.

\section{Families of isometries  in $\ell^2(\Z^d)$}\label{S-isometry}

The proof of   Theorem \ref{T1-N-interv} in dimension $d=1$  lead  G. Shaikh Samad  and the author of this paper  to the investigation of  a one-parameter   family of operators $\{T_t \}_{t\in\R}$  defined in $\ell^2=\ell^2(\Z)$   as follows: 
\begin{equation}\label{e2-defT}
(T_t(\vec a))_m= \begin{cases}  \dsize \frac{ \sin ( \pi t)}{ \pi}\sum_{  n\in\Z }      \frac{a_n  }{m-   n  +t} 
& \mbox{ if $t\not\in\Z$}
\cr  
 \dsize   (-1)^t a_{m+t} & \mbox{if $t \in\Z$}. \end{cases}
\end{equation}
See \cite{DS}.
When $t $ is not an integer   (and in particular when $t\in(-1,1)$),  these operators    can be viewed as discrete versions of the     Hilbert transform in $L^2(\R)$ (see \cite{Lae} and the reference therein). 
The main result in \cite{DS}  is the following:
   
\begin{Thm}\label{T2-semigr-H}
The family $\{T_t\}_{t\in\R}$ defined in \eqref{e2-defT} is a strongly continuous  Abelian group   of isometries  in $ \ell^2 $; its infinitesimal generator is $\pi H$, where  
$  (H(\vec a))_m= \frac 1 \pi \sum_{n\in\Z\atop n\ne m} \frac{a_n  }{m- {  n}  } 
$ 
is   the {\it discrete Hilbert transform}.  
\end{Thm}
 That is, we proved that
 $ T_s\circ T_t=T_{s+t}$;  that    for every $ \vec a\in \ell^2$,  the application  $t\to  T_t(\vec a)$   is continuous in $\R$;  that $||T_t(\vec a)||_{\ell^2} =||\vec a||_{\ell^2} $;   
 and finally  that, for every $\vec a\in \ell^2$,   
$
\lim_{t\to 0} \frac{T_t(\vec a)- \vec a}{t}=\pi H(\vec a), 
$ where the limit is in $\ell^2$.

Furthermore,   $T_t^*=T_{-t}=T^{-1}_t$  (see  Lemma 4.4 in \cite{DS}). 
\medskip

In this paper we   define multi-variable versions of the  operators $T_t$.   
We let $$\ell^2(\Z^d)=\{\vec a= \{a_{\vec n}\}_{\vec n\in\Z^d} \ : \  ||\vec a||_{\ell^2(\Z^d)}^2= \mbox{$ \sum_{n\in\Z^d} |a_{\vec n}|^2$}  <\infty\}.
$$
We denote  with  $\vec e_j$   the vector in $ \R^d$ whose components are all $=0$, with the exception of the $j-$th component  which is $=1$. 
 
Let $j\leq d$ and $s\in\R$; we define the operator 
 $T_{s\vec e_j}:\ell^2(\Z^d)\to\ell^2(\Z^d)$ as follows:
 \begin{equation}\label{e2-Td-multi}
 (T_{s\vec e_j}(  \vec a ))_{\vec n}=\begin{cases}  \dsize \frac{ \sin ( \pi s)}{ \pi}\sum_{m\in\Z }     
  \frac{a_{(n_1, ...,\,n_{j-1},\, m,\, n_{j+1},\,...n_d)}  }{n_j-  m  +s} 
& \mbox{ if $s\not\in\Z$}
\cr  
 \dsize   (-1)^s a_{(n_1,\, ...,\, n_j+s, ... n_d)} & \mbox{if $s \in\Z$}. \end{cases}
\end{equation}
For a given $\vec t=(t_1,\, ...,\, t_d)\in\R^d$, we let
 $T_{\vec t}  =T_{t_1\vec e_1}\circ...\circ T_{t_d\vec e_d}$.

For example,     when $\vec t \in (0, 1)^d $ we can easily verify that
$ 
(T_{\vec t} \,(\vec a))_{\vec m}=    \frac{\prod_{j=1}^d \sin ( \pi t_j)}{\pi^d}\sum_{\vec n\in\Z^d} \frac{a_{\vec n}  }{\prod_{j=1}^d(m_j-   n_j  +t_j)}.
$ 

\medskip
 The following multi-dimensional version of  Theorem \ref{T2-semigr-H} and Lemma 4.4 in \cite{DS} are easy to prove.
\begin{Thm}\label{T2-multi-T}
For every $\vec d,\ \vec t\in \R^d$ and every $\vec a\in \ell^2(\Z^d)$, 
\begin{equation}\label{e2-identities T}
T_{\vec d}\circ T_{\vec t}\,(\vec a)=  T_{\vec t}\circ T_{\vec d}\,(\vec a) =T_{\vec t+\vec d}\,(\vec a);
\end{equation}
  furthermore
\begin{equation}\label{e2-identities T2}
 T_{\vec t}^{-1}\,(\vec a)  =T_{-\vec t }\,(\vec a),
\end{equation}
and 
$$
||T_{\vec d}\, (\vec a)||_{\ell^2(\Z^d)} = ||\vec a||_{\ell^2(\Z^d)}.
$$
b) Let $ T_{\vec t} ^*$ be the adjoint of $T_{\vec t}$. Then, for every $\vec t\in\R^d$ and every $\vec a,\ \vec b\in \ell^2(\Z^d)$, we have that 
$$ T_{\vec t} ^*(\vec a)= T_{-{\vec t}}\,(\vec a)  $$
and
\begin{equation}\label{id-T*}
\l T_{\vec t}\,(\vec a),\ T_{\vec t }^*\,(\vec b) \r= \l \vec a, \ \vec b\r.
\end{equation}

\end{Thm}
%


\section{
Proofs}\label{S-proofs}

Let $Q= Q_1\cup...\cup Q_{N}$ be as in \eqref{defQ}.  By  \eqref{e2-frame} and \eqref{e2- Riesz-sequence},    the set $\B=\cup_{p=1}^{N} \{e^{2\pi i \vec x (\vec n+\vec \delta_p)}\}_{\vec n\in\Z^d}$  is a Riesz basis in $L^2(Q)$ if and only if there exists $A,\ B>0$ for which the following hold:
 \begin{equation}\label{e2-cond-span}
A||f||_{L^2(Q)}^2\leq \sum_{j=1}^{N}\sum_{\vec n\in\Z^d}\langle f, e^{ 2\pi i \l \vec n+\vec \delta_j,\,  x\r}\rangle^2_{L^2(Q)}  
 \leq B||f||_{L^2(Q)}^2 
\end{equation}
 whenever   $f \in L^2(Q)$, and 
\begin{equation}\label{e2-cond-lin-ind} A \sum_{j=1}^{N}\sum_{\vec n\in\Z^d} | a_{j,\vec n }|^2 \leq  \left\Vert \sum_{j=1}^{N}\sum_{\vec n\in\Z^d} a_{ j,\vec n } e^{ 2\pi i \l \vec n+\vec \delta_j , \,x\r}  \right\Vert_{L^2(Q)}^2\!\!\!\!\!\leq  B   \sum_{j=1}^{N}\sum_{\vec n\in\Z^d} | a_{\vec n\,j}|^2 \end{equation} 
for  every  finite  set of coefficients $\{\vec a_{j,\vec n}\}_{j\leq N\atop{\vec n\in\Z^d}}\subset\C$. 

\medskip
To prove Theorem \ref{T1-N-interv} we need the    following important

\begin{Lemma}\label{L2-inner-prod-Td}
Let $Q_{\vec M}=  Q_0+\vec M$, with $\vec M\in\R^d$.  Let $\vec a,\ \vec b\in\ell^2(\Z^d)$ and  $\vec t$, $\vec s\in \R^d$. Let $T_{\vec t}$ be   as in  Section 3.
Then, 
\begin{equation}\label{e2-inner-Sj} \l \sum_{\vec n\in\Z^d} a_{\vec n} e^{2\pi i\, \l \vec n+\vec s ,\, x\r },\ \sum_{\vec m\in\Z^d} b_{\vec m}e^{2\pi i\l \vec m +\vec t,\,   x\r}\r_{L^2(Q_{\vec M})}
\end{equation}
$$
=  e^{2\pi i \l\vec s -\vec t,\,  \vec M \r }\l  T_{\vec t}\,(\vec \alpha ), \   T_{\vec s}\,(\vec\beta)\r_{\ell^2(\Z^d)} 
$$
where

\noindent
  $\vec \alpha  = ((-1)^{n_1+...+n_d} e^{2\pi i \l \vec n,\, \vec  M\r } a_{\vec n})_{\vec n\in\Z^d}$, $\vec\beta = ((-1)^{m_1+...+m_d}e^{2\pi i\l \vec m,\, \vec  M\r}b_{\vec m})_{\vec m\in\Z^d}$.
 \end{Lemma}
 \noindent

 \begin{proof} 
 Assume $d=1$ since the proof  for $d>1$ is similar.  
 Assume first $s-t \not\in\Z$; We have
 
\begin{align} \nonumber
 &\l \sum_{n\in\Z}  a_{  n} e^{2\pi i\, ( n+ s )  x }, \ \sum_{m\in\Z}   b_{  m}e^{2\pi i(  m +  t)  x}\r_{L^2(M-\frac 12,M+\frac 12)}\\ \label{e-*}  & =
 \sum_{n,m\in\Z}a_n \bar b_m\int_{M-\frac 12}^{M+\frac 12}e^{2\pi i\, ( n-m+ s -t)  x }dx
 \\ \nonumber &
= e^{2\pi i ( s -t)  M } \sum_{n,m\in\Z} e^{2\pi i (n-m)  M }a_n \bar b_m\frac{ \sin( \pi (n-m+s -t))}{ \pi  (n-m+s -t)}
\\  \nonumber &
=e^{2\pi i (s -t)  M }\frac{\sin(\pi (s -t))}{\pi}\sum_{n,m\in\Z}\frac{\alpha_n \bar \beta_m}{  n-m+s -t }
\\ \nonumber & = e^{2\pi i (s -t)  M }\l T_{s -t}(\vec \alpha), \ \vec \beta\r_{\ell^2(\Z)}. 
\end{align}
The identity \eqref{e2-identities T} in Theorem \ref{T2-multi-T}   yields   
\eqref{e2-inner-Sj}.

When $s-t\in\Z$,  the integral in \eqref{e-*} equals $1 $ if $m=n +s-t $ and   is $=0$    in all other cases. We can 
write  
\begin{align*}
\eqref{e-*}=  &  \sum_{n \in\Z}a_{n } \bar b_{n+s-t}= (-1)^{s-t}e^{2\pi i (s -t) M   }\sum_{n \in\Z}\alpha_{n } \bar \beta_{n+s-t}\\ = & (-1)^{s-t} e^{2\pi i (s -t) M   }\l \vec \alpha, T_{t-s}(\vec \beta)\r.
\end{align*}
as required.
 \end{proof}

\subsection{ Proof of Theorem \ref{T1-N-interv}  }
 
 Let ${\bf \Gamma}={\bf \Gamma}(\l \vec M_p,\, \vec \delta_j\r)$ be as in \eqref{defGamma}. 
We  show  that \eqref{e2-cond-span} holds if and only if $\det{\bf \Gamma}\ne 0$.     
Let $f\in L^2(Q)$.
 For every $j\leq N$, we obtain
\begin{align}\nonumber
\sum_{ \vec n\in\Z^d}  &    | \l f, e^{ 2\pi i \l \vec n+\vec \delta_j,\, x\r}  \r_  {L^2(Q)}|^2 = \sum_{ \vec n\in\Z^d}      \left | \sum_{p=1}^{N}\l f, e^{ 2\pi i \l \vec n+\vec \delta_j,\, x\r}\r_{L^2( Q_0+\vec M_p)}\right |^2
    \\ \label{e3-frame-prod}& 
 =  \sum_{\vec n\in\Z^d } \sum_{p,  q=1}^{N}\l f, e^{ 2\pi i \l \vec n+\vec \delta_j,\, x\r}\r_{L^2( Q_0+\vec M_p)}\overline{\l f, e^{ 2\pi i \l \vec n+\vec \delta_j,\, x\r}\r_{L^2( Q_0+\vec M_q)}}.
\end{align}
  Observe that 
 \begin{align}\nonumber
 {\l f, e^{ 2\pi i \l \vec n+\vec \delta_j,\, x\r}\r_{L^2( Q_0+\vec M_p)}} &=\int_{ Q_0+\vec M_p }  f(x) e^{ -2\pi i \l \vec n+\vec \delta_j,\, x\r}dx \\ \nonumber &=
  \int_{Q_0}   f(y+\vec M_p) e^{ -2\pi i \l \vec n+\vec \delta_j,\,  y+\vec M_p\r}dy
  \end{align}
  \begin{align}\nonumber &
  =
  e^{-2\pi i \l\vec \delta_j,\,\vec M_p\r} \int_{Q_0}  \tau_{\vec M_p}f(y)  e^{ -2\pi i \l \vec n+\vec \delta_j,\, y\r}dy 
 \\  \label{e4-eq33}& = e^{-2\pi i \l\vec \delta_j,\,\vec M_p\r} \l \tau_{\vec M_p}f , \ e^{  2\pi i \l \vec n+\vec \delta_j,\, x\r}\r_{L^2(Q_0)}    
  \end{align}
 where $ \tau_{\vec v}f(x)= f(x+\vec v)$. Since the set $\{e^{2\pi i \l n+\delta_j, x\r}\}_{\vec n\in\Z^d}$ is an orthonormal basis on $L^2(Q_0)$, by Parseval identity and \eqref{e4-eq33}   the sum in \eqref{e3-frame-prod} equals   
\begin{equation}\label{e4-44}  
 \sum_{p, q=1}^{N}\!e^{ 2\pi i \l \vec\delta_j,\, \vec M_q-\vec M_p\r}\!\!\sum_{n\in\Z^d }  \l \tau_{\vec M_p}f  ,   e^{  2\pi i \l \vec n+\vec \delta_j,\, x\r}\r_{L^2(Q_0)}\overline{\l \tau_{\vec M_q}f ,   e^{  2\pi i \l \vec n+\vec \delta_j,\, x\r}\r_{L^2(Q_0)}  }
 $$$$=
 \sum_{p, q =1}^{N }e^{ 2\pi i \l \vec\delta_j,\, \vec M_q-\vec M_p\r} \l  \tau_{\vec M_p} f,\     \tau_{\vec M_q} f\r_{L^2(Q_0)}  
\end{equation}
 and  from \eqref{e3-frame-prod}  and \eqref{e4-44} follow that
\begin{equation}\label{e3-positive-sum}
 \sum_{j=1}^{N} \sum_{\vec n\in\Z^d}    |\langle f, e^{ 2\pi i \l\vec n+\vec \delta_j,\,  x\r}\rangle_{L^2(Q)}|^2 =
 \sum_{p, q =1}^{N } \beta_{p,q} \l  \tau_{\vec M_p} f,\     \tau_{\vec Mq} f\r_{L^2(Q_0)}
\end{equation}
 where we have let
 \begin{equation}\label{def-beta}
 \beta_{p,q}= \sum_{j=1}^{N}  e^{ 2\pi i \l \vec\delta_j,\, \vec M_q-\vec M_p\r},\quad 1\leq p,q\leq N.
 \end{equation}
 Let $ {\bf B}$ be the $N\times N$ matrix whose elements are the $\beta_{p,q} $.  We can let $\gamma_{j,p} =e^{ 2\pi i \l \vec\delta_j,\, \vec M_p\r}$  and  write   $\beta_{p,q}= \sum_{j=1}^{N} \overline{\gamma_{j,p}}\gamma_{j,q}$; thus,  ${\bf B}= {\bf \Gamma}^*{\bf \Gamma }$ where  ${\bf \Gamma}=\{\gamma_{j,p}\}_{1\leq j,p\leq N}$ is as   in  \eqref{defGamma} and  ${\bf \Gamma^*}$ is the   conjugate   transpose of ${\bf \Gamma}$.   
 
 The maximum (minimum) value of the Hermitian form $\l {\bf B}\vec v,\, \vec v\r$ on   $\S_\C^{N-1}$, the unit sphere of $\C^N$,  equal   the  maximum (minimum) eigenvalue of ${\bf B}$.
  %
 Let $\Lambda=\max_{x\in\S^{N-1}}\l {\bf B} x,\, x\r$, and $\lambda=\min_{x\in\S^{N-1}}\l {\bf B} x,\, x\r$. We have 
$$
\sum_{p, q =1}^{N } \beta_{p,q} \l  \tau_{\vec M_q} f,\     \tau_{\vec M_p} f\r_{L^2(Q_0)}
=
\int_{Q_0} \sum_{p, q =1}^{N } \beta_{p,q} f(x+\vec M_p)\overline{f(x+\vec M_q)}dx
$$
$$
\leq 
\Lambda \int_{Q_0} \sum_{p  =1}^{N }  |f(x+\vec M_p)|^2 dx=
\Lambda \sum_{p  =1}^{N } \int_{ Q_0+\vec M_p}   |f(x )|^2 dx= \Lambda ||f||_{L^2(Q)}^2.
$$
Similarly, we prove that $\sum_{p, q =1}^{N } \beta_{p,q} \l  \tau_{\vec M_q} f,\     \tau_{\vec M_p} f\r_{L^2(Q_0)}\ge \lambda  ||f||_{L^2(Q)}^2$.  
From \eqref{e3-positive-sum} and the  inequalities above follows \eqref{e2-cond-span},  with $A=\lambda$ and $B=\Lambda$.  

\medskip
We prove that 
$A=\lambda$ and $B= \Lambda$ are  the  optimal frame bounds of $\B$.  

Let $\vec w=(w_1, ...,\, w_d)$ be  an  eigenvector  of ${\bf B}$ with eigenvalue $\Lambda$; let   $g(x)=\sum_{q=1}^N w_q\chi_{Q_0-\vec M_q}(x)$, where $\chi $ is the characteristic function. The identity \eqref{e3-positive-sum} yields
\begin{align*}
\sum_{j=1}^{N} \sum_{\vec n\in\Z^d}  &  |\langle g, e^{ 2\pi i \l\vec n+\vec \delta_j,\,  x\r}\rangle_{L^2(Q)}|^2 = 
\int_{Q_0} \sum_{p, q =1}^{N } \beta_{p,q}\ \tau_{\vec M_p}g(x)\overline{\tau_{\vec M_q}g(x)}dx
\\ = &
\sum_{p, q =1}^{N } \sum_{h, k =1}^{N }\beta_{p,q}\,\omega_k\bar\omega_k \int_{Q_0} \chi_{Q_0-\vec M_h}(x+\vec M_p)\chi_{Q_0-\vec M_k}(x+\vec M_q) dx
\\ = &
 \sum_{p, q =1}^{N } \beta_{p,q}\ w_p\bar w_q=\Lambda \sum_{p =1}^{N } |w_p|^2=\Lambda||g||_{L^2(Q)}^2.
\end{align*}
 A similar argument shows that $A=\lambda$ is the optimal lower frame bound of $\B$.

\bigskip
 \noindent 
Let us prove  that  \eqref{e2-cond-lin-ind} holds if and only if $\det {\bf \Gamma} \ne 0$.   Let $\vec a_1$, ...,\, $\vec a_N\in\ell^2(\Z^d)$ be such that $a_{j,\vec n}=0$  for all but a finite set of $\vec n\in\Z^d$.    Let 
  $  \dsize S_j  =\sum_{\vec n \in\Z^d} a_{j, \vec {n }} e^{2\pi i\, \l\vec n +\vec \delta_j,\,  x\r }  .$ 
  The central sum in \eqref{e2-cond-lin-ind} equals $||S_1+...+S_{N}||^2_{L^2(Q)} $. By Lemma \ref{L2-inner-prod-Td}
\begin{align}\nonumber
||S_1+...+S_{N}||^2_{L^2(Q)}  & =   \sum_{1\leq i, j\leq N}\langle S_i,\ S_j\rangle_{L^2(Q)}\\ \nonumber &=
\sum_{1\leq i, j\leq N} \sum_{p=1}^{N}  \langle S_i,\ S_j\rangle_{L^2( Q_0+\vec M_p)}.
\\  \label{e3-basic-exp} 
 &=  \alpha_{i,j}\l T_{\vec \delta_i }(\vec a_i),  T_{ \vec \delta_j}(\vec a_j)\r_{L^2(\Z^d)} 
\end{align}
where  we have let
\begin{equation}\label{def-alpha}
\alpha_{i,j}= \sum_{p=1}^{N}e^{2\pi i \l\vec \delta_i-\vec \delta_j,\ \vec M_p \r}  
\end{equation}
Let $\A$ be the $N\times N$ matrix whose elements are  the $\alpha_{i,j}$.  It is easy to verify  that  $\A= {\bf \Gamma}{\bf \Gamma^*}$, where ${\bf \Gamma}$ is as in \eqref{defGamma}. 
The  matrices $\A={\bf \Gamma}{\bf \Gamma^*} $ and ${\bf B}={\bf \Gamma^*}{\bf \Gamma} $  have the same eigenvalues (see Lemma \ref{L2-eigen-TT*} below) and for every $\vec v\in\C^N$,
\begin{equation}\label{int11}\lambda  ||\vec v||^2\leq   \sum_{i,j=1}^{N} \alpha_{i,j}v_i \overline v_j\leq \Lambda ||\vec v||^2
\end{equation}
where $\Lambda$ and $\lambda$ are     as in the first part of the proof.
 
Let   $T_{\vec \delta_k }(\vec a_k) =  \vec b_k$. In view of  \eqref{e3-basic-exp},
\begin{equation}\label{int1}
||S_1+...+S_{N}||^2_{L^2(Q)}=\sum_{i,j=1}^{N} \alpha_{i,j}\l  \vec b_i, \vec b_j \r_{\ell^2(\Z^d)}=\sum_{\vec n\in\Z^d}\sum_{i,j=1}^{N} \alpha_{i,j}b_{i, \vec n}\overline b_{j, \vec n}.
\end{equation}
Fix $\vec n\in\Z^d$;
by \eqref{int11},
$$\lambda  (|b_{1, \vec n}|^2+...+|b_{N , \vec n}|^2)\leq \sum_{i,j=1}^{N} \alpha_{i,j}b_{i, \vec n}\overline b_{j, \vec n}\leq \Lambda  (|b_{1, \vec n}|^2+...+|b_{N, \vec n}|^2)
$$
and when we sum with respect to $\vec n$ we obtain  
\begin{equation}\label{int2}
\lambda   \sum_{j=1}^N ||\vec b_j||^2_{\ell^2(\Z^d)}\leq \sum_{\vec n\in\Z^d}\sum_{i,j=1}^{N} \alpha_{i,j}b_{i, \vec n}\overline b_{j, \vec n}\leq \Lambda  \sum_{j=1}^N ||\vec b_j||^2_{\ell^2(\Z^d)}. 
\end{equation}
Recalling that $\vec b_j=T_{\vec \delta_j}(\vec a_j)$ and that the $T_{\vec \delta_j}$   are invertible isometries,  we obtain $||\vec b_j||_{\ell^2(\Z^d)}=  ||\vec a_j||_{\ell^2(\Z^d)}$.
 From   \eqref{int1} and \eqref{int2} follows that
 $$
\lambda  \sum_{j=1}^N ||\vec a_j||^2_{\ell^2(\Z^d)} \leq ||S_1+...+S_{N}||^2_{L^2(Q)} \leq \Lambda  \sum_{j=1}^N ||\vec a_j||^2_{\ell^2(\Z^d)}. 
$$
and \eqref{e2-cond-lin-ind}   is proved with $ B=\Lambda$ and $A=\lambda$. By Lemma \ref{L-same-consts}, $\Lambda$ and $ \lambda$  are the optimal constants in \eqref{e2-cond-lin-ind}.
The proof of Theorem \ref{T1-N-interv} is completed. $\Box$
 
  \bigskip

  \noindent
  {\it Remark.}     The proof of Theorem \ref{T1-N-interv} shows that \begin{align*}
  &\mbox{$\B$ is a frame $\iff \det   {\bf B}\ne 0 \iff \det{\bf \Gamma}\ne 0$} 
  \\ & \mbox{$\iff \det\A\ne 0 \iff \B$ is a Riesz sequence}
  \end{align*}
  This observation proves Theorem \ref{Cor-equiv-bases-frames}.
  
  \medskip 
  For the sake of completeness, we  prove the following easy 
 \begin{Lemma}\label{L2-eigen-TT*} Let $\M$ be a  square   matrix.
 The matrices  ${\  M} {  M^*}$ and $ { M^*} {  M }$  have
 the same eigenvalues. 
 \end{Lemma}
 
 \begin{proof}
 Let  $\lambda$ be an eigenvalue of ${\  M} {  M^*}$ with eigenvector $\vec v$.  Thus,  ${  M} {  M^*}\vec v=\lambda \vec v$, and if we multiply this identity to  the  right by $  {  M^*} $, we obtain
 $( {  M^*}{  M} ){  M^*}\vec v=\lambda {  M^*}\vec v.$  Thus, $\lambda $ is an eigenvalue of ${\  M^*}{  M} $ with eigenvector ${   M^*}\vec v$.
 \end{proof}
 
 \begin{proof}[Proof of Theorem \ref{C-meas=0}]
 
For a given the multi-rectangle $Q=Q(\vec M_1,\, ...,\, \vec M_N)$ as in \eqref{defQ}, we let   $$Z=\{ (\vec\delta_1, ...,\vec\delta_N)\in (\R^d)^N\ : \ \det {\bf \Gamma}(\l \vec M_p,\, \vec \delta_j\r) =0\}. 
$$ By Theorem \ref{T1-N-interv}, $\B(\vec\delta_1, ...,\vec\delta_N) $ is  a Riesz basis of $L^2(Q)$ if and only if $(\vec\delta_1, ...,\vec\delta_N)\not\in Z$.
Consider the function
$\psi:(\C^d)^N\to\C$,
$ \psi (\vec\eta_1, ...,\vec\eta_N)=  \det{\bf   \Gamma}(\l \vec m_p,\, \vec \eta_j \r). $
Since $\psi$ is holomorphic,   by   \cite[ Corollary 10]{GR},  the set where $\psi\equiv 0$  has 
 zero Lebesgue measure  in $(\C^d)^N$. Thus,   $Z$ has zero Lebesgue measure in $\R^d$.
\end{proof}
 \subsection{Proof of Theorem \ref{T1-special-delta}}

\medskip
 Let $\s(\vec \delta)$  be as in \eqref{d2-def-S}.
 By Theorem \ref{T1-N-interv}, $\s(\vec \delta)$ is a Riesz basis of $L^2(Q)$ if and only if the matrix ${\bf \Gamma}=\{\gamma_{j,p}\}_{1\leq j,\,p\leq n}$, with   $ \gamma_{j,p}=  e^{2\pi i   \l \vec M_p,\,  (j-1)\vec \delta\r} $   is non-singular. 
Since $\gamma_{j,p}=  \left(e^{2\pi i \l \vec M_p,\,  \vec \delta\r}\right)^{j-1}$,   ${\bf \Gamma}$ is a {\it Vandermonde matrix}, i.e., a  matrix with the terms of a geometric progression in each row. 
We have 
  $$\det {\bf \Gamma}=\prod_{p< q} \left(e^{2\pi i \l \vec M_p,\,  \vec \delta\r}-e^{2\pi i \l \vec M_q,\,  \vec \delta\r}\right), 
$$
and  so ${\bf \Gamma}$ is non-singular  if and only if   $e^{2\pi i \l \vec M_p,\,  \vec \delta\r}-e^{2\pi i \l \vec M_q,\,  \vec \delta\r}\ne 0$ whenever $p\ne q$.  Equivalently, ${\bf \Gamma}$ is non-singular  if and only if $\l\vec M_q-\vec M_p,\,  \vec \delta\r\not\in\Z$ whenever $p\ne q$.

By Theorem \ref{T1-N-interv} and \eqref{def-beta}, the frame constants  of $\s(\vec \delta)$ are
  the minimum and maximum eigenvalue of   $ {\bf \Gamma^*\Gamma}={\bf   B}=\{\beta_{p,q}\}_{1\leq p,q\leq N}$, where 
  \begin{align*} \beta_{p,q}= &
  \sum_{j=1}^{N}  e^{ 2\pi i (j-1)\l \vec\delta ,\, \vec M_q-\vec M_p\r}=\frac{1-e^{ 2\pi i N\l \vec\delta ,\, \vec M_q-\vec M_p\r}}{1-e^{ 2\pi i  \l \vec\delta ,\, \vec M_q-\vec M_p\r}}
\\
= & e^{  \pi i (N-1)\l \vec\delta ,\, \vec M_q-\vec M_p\r} \frac{\sin(\pi N \l \vec\delta ,\, \vec M_q-\vec M_p\r)}{
\sin(\pi   \l \vec\delta ,\, \vec M_q-\vec M_p\r)}.
\end{align*}
Let   ${\bf \tilde B}$ be the matrix whose elements are $\tilde \beta_{p,q}=\frac{\sin(\pi N \l \vec\delta ,\, \vec M_q-\vec M_p\r)}{
\sin(\pi   \l \vec\delta ,\, \vec M_q-\vec M_p\r)}$ when $p\ne q$ and $\tilde \beta_{p,q}=N$ when $p=q$;  
we have  
$   \l {\bf B} \vec v, \ \vec v\r=   \l {\bf \tilde B} (v\odot \vec w),\ (\vec v \odot \vec w) \r $, 
where     $\vec w= (e^{  \pi i (N-1)\l \vec\delta ,\, \vec M_1\r} ,\ ...,\,\ e^{  \pi i (N-1)\l \vec\delta ,\, \vec M_N\r} )$ and $\odot$ denotes the Hadamard (componentwise) product.

Since $||\vec w||= ||\vec v||$, 
the Hermitian forms  $\vec v\to\l {\bf B} \vec v, \ \vec v\r$  and $\vec v\to    \l {\bf \tilde B} \vec v,\, \vec v\r$  have the same maximum and minimum on    $\S^{N-1}_{\C}$; consequently, 
 the optimal frame bounds of $\s(\vec\delta)$ are the maximum and minimum eigenvalues of  ${\bf \tilde B}$, as required.  $\Box$
  
  \medskip
  \noindent
  {\it Remark.} It is interesting to observe that 
 
\begin{align*} 
 |\det{\bf \Gamma}|^2 = \det {\bf B} = & \prod_{p< q} \left|e^{2\pi i \l \vec M_p,\,  \vec \delta\r}-e^{2\pi i \l \vec M_q,\,  \vec \delta\r}\right|^2\\ &
=2^{\frac{N(N-1)}{2}} \prod_{p< q}(1-\cos(2\pi\l\vec M_p-\vec M_q,\,  \vec \delta\r))\\ & = 2^{N(N-1)} \prod_{p< q} \sin^2(\pi\l\vec M_p-\vec M_q,\,  \vec \delta\r).
\end{align*}
 \section{Corollaries and examples}\label{S-Corollaries}

In this section we prove a number of corollaries of  Theorems \ref{T1-N-interv} and \ref{T1-special-delta}. We   use the following notation:  
 we denote with $\vec a \odot \vec b = a_1b_1+...+a_db_d$  the Hadamard product of  the vectors   $\vec a=(a_1,\, ...,\, a_d),\ \vec b=(b_1,\, ...,\,b_d)\in\C^d$; if
   $A$ is a nonempty subset of $\R^d$ and $\vec v\in \C^d$,  we let     
    $A\odot \vec v= \{  x\odot \vec v\ : \ x\in A\}$.

In this section we will  always assume, often without saying, that $N$ denotes  a positive integer.  We start with a corollary of Theorem \ref{T1-N-interv}.
 
  \begin{Cor}\label{L-[0,N]}
Let $I=[a, a+N] $, with $a\in\R$ and $N\ge 1 $.  Let $\delta_1 ,\, ...,\, \delta_N\in\R$ and  %
$\B=\B(\delta_1, \, ...,\,\delta_N)$  as in \eqref{d2-defofB}. Then, $\B$ 
is a Riesz basis on $L^2(I)$   if and only if  $\delta_i-\delta_j\not\in\Z$. The optimal frame constants  of $\B$  are     the maximum and minimum eigenvalues of the matrix ${\bf \tilde A} =\{
\tilde\alpha_{i,j}\}_{1\leq i,j\leq N}$, where  $$\tilde\alpha_{i,j}=\begin{cases}  \frac{\sin(\pi N(\delta_i-\delta_j))}{ \sin (\pi (\delta_i-\delta_j))} & \mbox{ if $i\ne j $ }, \cr N &\mbox{ if $i=j$}. \end{cases}$$
\end{Cor}

\begin{proof}
Without loss of generality, we can let  
 $I= [-\frac 12, N-\frac 12)=   \cup_{p=0} ^{N-1} [p-\frac 12, p+\frac 12)$; by Theorem \ref{T1-N-interv},   $\B(\delta_1, \, ...,\,\delta_N)$ is a Riesz basis of $L^2(I)$ if and only if  the matrix ${\bf \Gamma}=\{ e^{2\pi i p \delta_j}\}_{0\leq j,p\leq N-1}$ is nonsingular.  But $\ e^{2\pi i p \delta_j}=\left(e^{2\pi i  \delta_j} \right)^{p}$, and so ${\bf \Gamma}  $ is a Vandermonde matrix whose  determinant is $ \det{\bf\Gamma}= \prod_{0\leq i<j=N-1} (e^{ 2\pi i  \delta_i}-e^{ 2\pi i  \delta_j})$. Therefore, $ \det{\bf\Gamma}\ne 0\iff \delta_i-\delta_j\not\in\Z.
 $ 
In view of \eqref{def-alpha}, the  optimal  frame constant of ${\cal B}$ are the maximum and minimum eigenvalue of the matrix $\A=\{\alpha_{i,j}\}_{1\leq i,j\leq N}$, with $ \alpha_{i,j}= \sum_{p=0}^{N-1}e^{2\pi i( \delta_i-  \delta_j) p } $.
We can easily verify that $\alpha_{i,j}=e^{ \pi i(N-1)( \delta_i-  \delta_j)  }\frac{\sin( \pi N( \delta_i-  \delta_j) ) }{\sin( \pi ( \delta_i-  \delta_j))  }$ when $i\ne j$, and $\alpha_{i,j}=N$ when $i=j$; if we argue as in the second part of the proof of Theorem \ref{T1-special-delta}, we can conclude that the frame constants of ${\cal B} $ are  the maximum and minimum eigenvalue of the matrix ${\bf \tilde  A}$ defined above.

\end{proof}

 \noindent
 {\it Remark}. In view of \cite[Lemma 2.1]{SZ} (see Section 2.2)  it is not too difficult to prove a multi-dimensional version of Corollary \ref{L-[0,N]}. We leave the details to the interested reader.
 
 \subsection{Multi-rectangles with vertices in  $\Q^d$ and a stability theorem}
 
 Let $  R$ be a multi-rectangle  with  vertices in $\Q^d$. 
Let  $ l_1, ...,l_d $ be  the smallest positive integers for which   $\tilde R= R\odot  (l_1,...,l_d) $  has vertices in $\Z^d$. Let $\vec l= (l_1, \, ...,\, l_d)$ and $L=  \prod_{j=1}^N l_j$.
The multi-rectangle  $\tilde R$   is a  union of $N=L|R|$ disjoint unit cubes with vertices in $\Z^d$.  Let $\vec\delta_1,\, ...,\, \vec\delta_N\in\R^d$ and $ \B(\vec\delta_1,\, ...,\, \vec\delta_N) $ as in \eqref{d2-defofB}.
 A  simple scaling   in   \eqref{e2-cond-span} and \eqref{e2-cond-lin-ind}  proves the following 
 \begin{Cor}\label{Cor-rational-cubes}
  $  \B(\vec\delta_1,\, ...,\, \vec\delta_N)$ is a Riesz basis of $L^2(\tilde R)$ with constant $A$ and $B$ if and only if  $ \tilde\B=\bigcup_{j=1}^N \{ e^{2\pi i  \l  (\vec n+\vec\delta_j)\odot \vec l,\   x\r }\}_{n\in\Z^d}$ 
 is a Riesz basis of $L^2(R)$ with constants $ \frac A L $ and $ \frac B L $.   

 \end{Cor}

\medskip
Our next result is a  stability theorem for the basis $E(\Z)$ on $L^2(0,1)$:
\begin{Cor}\label{C-Kadec}
Let  ${\cal E}=\{\epsilon_j\}_{j\in\Z}\subset  \R$; assume that, for some integer  $N\ge 2$, we have that  
for which  \begin{equation}\label{e5-cond-Kadec1}
 \epsilon_j = \epsilon_{N+j},\ \quad j\in\Z  
\end{equation}
The set  ${\cal U}=\{e^{2\pi i(n+\epsilon_n)x}\}_{n\in\Z}$ is a Riesz basis of $L^2(0,1)$ if and only if
\begin{equation}\label{e5-cond-Kadec}
\frac{\epsilon_i -\epsilon_j +i-j}{N}\not\in\Z\ \mbox{\ whenever \ \  $  i\ne j $.} 
\end{equation}
In particular, ${\cal U}$ is a Riesz basis of $L^2(0,1)$  whenever $ \epsilon_i -\epsilon_j\not\in\Z$.
\end{Cor}

Kadec theorem implies that  a set $E(\Lambda)=\{e^{2\pi i \lambda_n x}\}_{n\in\Z}$  is a Riesz basis  of $L^2(0,1)$ whenever   $\ell=\sup_{j\in\Z} |j-\lambda_j|<  \frac 14$.   When    $\ell\ge   \frac 14$,   $E(\Lambda)$ may  not be  a Riesz basis  of  $L^2(-\frac 12, \frac 12)$,  but  that is the case if the $\epsilon_j=\lambda_j-j$ satisfy  \eqref{e5-cond-Kadec} and 
 \eqref{e5-cond-Kadec1}.

 
 \medskip
 
 We prove first the prove the  following

\begin{proof}[Proof  of Corollary \ref{C-Kadec}]
  
By  \eqref{e5-cond-Kadec1}, $\epsilon_k=\epsilon_m$  whenever   $m=nN+k $    for some $n\in\Z$.
Thus, \begin{align*}
{\cal U}=   \{ e^{2\pi i (m+\epsilon_m) x}\}_{m\in\Z} 
=  
\bigcup_{j=0}^{N-1} \{ e^{2\pi i (N n+j  +\epsilon_{j } ) x}\}_{n\in\Z}  =\bigcup_{j=0}^{N-1} \{ e^{2\pi i N(n+\frac{j+\epsilon_{j}}{N} ) x}\}_{n\in\Z}. \end{align*}
By Corollary \ref{Cor-rational-cubes},    ${\cal U}$
is a Riesz basis of $L^2(0,1)$ if and only if  the set $\tilde {\cal U}=\bigcup_{j=0}^{N-1} \{ e^{2\pi i  (n+\frac{j+\epsilon_{j }}{N} ) x}\}_{n\in\Z} $ is a Riesz basis of $L^2(0,N)$. We can apply Corollary \ref{L-[0,N]},  with $\delta_j=\frac{j+\epsilon_{j }}{N}$, and conclude that $\delta_j-\delta_j\not\in\Z$ is equivalent to \eqref{e5-cond-Kadec}.  
\end{proof}
 
 \medskip
 \noindent
 {\it Example 1.} Fix   $s\in (0,1)$, and define $ {\cal U}=\{ e^{2\pi i  ( m+ \mu_m ) x}\} _{m\in\Z}$
 where $\mu_m=s$ when $m$ is even and $\mu_m=-s$ when $m$ is odd. 
   By 
  \eqref{e5-cond-Kadec},   $ {\cal U}$ is a Riesz basis of $L^2(0,1)$ if and only if $\frac{ 2s-1}{2}\not\in\Z$, or 
 $0<s<\frac 12$.   
 
\medskip
 It is interesting to compare Example 1  with a famous example 
 by Ingham. In  \cite[p. 378]{I}  it is proved   that  ${\cal U} $  is not a Riesz basis of $L^2(0,1)$ when $\mu_m=\frac 14$ if $m>0$, $\mu_m =-\frac 14$ if $m<0$ and $\mu_0=0$.  Ingham's example shows that the constant $\frac 14$ in Kadec's theorem cannot be replaced by any larger constant. See also \cite{Y}.
 
 \medskip
 \noindent

 \subsection{Two cubes in $\R^d$}
  
  Let  $\vec M_1\ne \vec M_2\in\Z^d$ and let $Q = \tau_{\vec M_1}Q_0\cup \tau_{\vec M_2}   Q_0   $.     Let $\vec\delta_1,\ \vec\delta_2\in\R^d$.  
  We prove the following 
 
\begin{Cor}\label{C4-2 cubes} a) The set $\B=\B(\delta_1, \,\delta_2)$ is a Riesz basis of $L^2(Q )$ if and only if   $\l\vec M_1-\vec M_2, \,  \vec \delta_1-\vec\delta_2\r\not\in\Z$. The optimal frame constants of $\B$ are 
$$A=2(1-|\cos(\pi \l\vec M_1-\vec M_2,\,  \vec \delta_1-\vec\delta_2\r)  |),\  B=2(1+|\cos(\pi \l\vec M_1-\vec M_2,\,    \vec \delta_1-\vec\delta_2\r)  |).$$

 In particular, $\B $ is  an orthogonal Riesz basis of $L^2(Q)$  if and only if   $$ 2 \l \vec M_1-\vec M_2,\    \vec \delta_1-\vec\delta_2\r \in\Z,\ \mbox{ and}\  \l \vec M_1-\vec M_2,\,   \vec \delta_1-\vec\delta_2\r \not\in\Z.
$$
\end{Cor}

 \begin{proof}    
After perhaps a translation, we can let   $Q=Q_0\cup+ \tau_{\vec M}Q_0 $, with $\vec M=\vec M_2-\vec M_1$.  
By Theorem \ref{T1-N-interv},  $\B$ is a Riesz basis if and only if the matrix 
$$
{\bf A}=\left(\begin{matrix}  2, & 1+e^{2\pi i \l \vec M,\,\vec \delta_1-\vec\delta_2\r }
\\ 
1+ e^{-2\pi i \l \vec M,\,\vec \delta_1-\vec\delta_2\r },  & 2 
\end{matrix}\right)
$$
is nonsingular.  The eigenvalues of $\A$ are  the zeros of the characteristic polynomial,
$$
\mbox{det}({\bf A}-s {\bf I})=(2-s)^2-\left|1+e^{2\pi i \l \vec M,\,  \vec \delta_1-\vec\delta_2\r }\right|^2  
$$
where ${\bf I}=\left(\begin{matrix}  1  & 0\\ 0  & 1\end{matrix}\right)$. 
We can easily verify that
$$
\mbox{det}({\bf A}-s {\bf I}) = 
 s^2-4s +4\sin^2 (\pi \l\vec M,\,  \vec \delta_1-\vec\delta_2\r )=0 \iff
$$
$s= 2(1\pm |\cos(\pi \l\vec M\,  \vec \delta_1-\vec\delta_2\r )|) $. Thus, $\lambda=2(1- |\cos(\pi \l\vec M\,  \vec \delta_1-\vec\delta_2\r )|)$ and $\Lambda=2(1+ |\cos(\pi \l\vec M\,  \vec \delta_1-\vec\delta_2\r )|)$  are the optimal frame constants of $\B$.

    \medskip
  When $\cos(\pi \l\vec M\,  \vec \delta_1-\vec\delta_2\r)=0$, i.e. when  $\l\vec M,\, \vec \delta_1-\vec\delta_2 \r$ is an odd multiple of $\frac 12$,  
  $\B$ is a tight frame with constants   $\lambda=\Lambda=2$.
  Since the functions in $\B$ have norm $=2$ on $L^2(Q)$,   Lemma \ref{L-normalized-tight} implies that $\B$ is orthogonal.  

\end{proof}

  \noindent
  {\it Remark.}  
Let $Q$ be 
the union  of two disjoint unit cubes with vertices in $\Z^d$. We can verify that $Q$ tiles $\R^d$ by translation;   by
  Corollary \ref{C4-2 cubes}, we can always find an orthogonal basis on $L^2(Q)$ and so  $Q$ is a spectral domain of  $\R^d$.
It is  proved in  \cite{La}  that     the union of two  disjoint intervals of nonzero length  is   spectral   if and only it   tiles  $\R$ by translation. To the best of our knowledge,   the  analog of the main theorem in \cite{La}  has not been proved  (or disproved) for  unions of two disjoint  rectangles in $\R^d$.

\subsection{Spectral domains in $\R^d$}
In this section we show examples of  spectral  multi-rectangles  in $\R^d$.   We let   $Q =Q(\vec M_1,\, ...,\, \vec M_N)\subset \R^d$ be as in \eqref{defQ}  and $ \s(\vec\delta)   =\bigcup_{j=1}^{N-1}\{e^{2\pi i  \l \vec n+(j-1)\vec \delta,\, x\r }\}_{\vec n\in\Z^d} 
 $ be as in \eqref{d2-def-S}.   The following Corollary can be viewed as a generalization of Corollary \ref{C4-2 cubes}.

 \begin{Cor}\label{Cor-orthonormal}
  The set  $\s(\vec\delta)$   is an orthogonal basis of $L^2(Q) $ 
if and only if, for every $p\ne q$,   
\begin{equation}\label{e-cor-orthon-N} \l\vec M_p-\vec M_q,\,  \vec \delta\r\not\in\Z\  \mbox{ and }\  N\l\vec M_p-\vec M_q,\,  \vec \delta\r\in\Z.
\end{equation}  
 
 \end{Cor}
 \begin{proof}  Assume that \eqref{e-cor-orthon-N} holds.  By Theorem \ref{T1-special-delta}, the first condition in \eqref{e-cor-orthon-N} yields that  $\s(\vec \delta)$ is a Riesz basis of $L^2(Q)$; the second condition  implies  
   the matrix ${\bf \tilde B}$ as in \eqref{def-betatilde-Thm13}
 is diagonal. Thus,     $\s(\vec \delta)$ is a   tight frame with frame constant $N=|Q|$, and by  
   Lemma \ref{L-normalized-tight},  it  is an orthogonal basis of $L^2(Q)$.
  
 Conversely, assume that $\s(\vec \delta)$  is orthogonal.  Thus, the frame constant of $\s(\vec \delta)$ equal $A=B=N$ and  by Theorem \ref{T1-special-delta}, the maximum and minimum   eigenvalues  of the matrix ${\bf \tilde B}$ equal $N$ as well.   Since    $\C^N$ has a basis of eigenvectors of  ${\bf \tilde B}$, we can  infer that  ${\bf \tilde B}\vec v =N\vec v$ for every $\vec v\in\C^N$ and that 
  ${\bf \tilde B}$ is diagonal. Recalling that the elements of ${\bf \tilde B}$ are as in   \eqref{def-betatilde-Thm13}, we deduce \eqref{e-cor-orthon-N}.

 \end{proof}

 \noindent
 {\it Example 2.} Let $Q=Q(\vec M_1,\, ...,\, \vec M_N) $,   with  $N\leq d$. Assume that $Q_N=Q_0$ (so that $\vec M_N=(0,\, ...,\,0)$)  and that  the  $\vec M_1,\, ...,\, \vec M_{N-1}$'are linearly independent.  We show that  $ Q $ is  spectral.

Let 
 ${\bf M}$ be the matrix whose rows are $\vec M_1$,\, ...,\, $\vec M_{N-1}$.  By assumption,    ${\bf M}$ has rank $N-1$, and so we can find  $\vec \sigma \in\R^d $  that satisfies 
$ 
 \l\vec M_j ,\,  \vec \sigma\r= \frac jN $ for every $ j=1, ...,\, N-1.
 $ 
By  Corollary \ref{Cor-orthonormal},   $\s(\vec\sigma)$ is an orthogonal basis of $L^2(Q)$.

 \subsection{Extracting Riesz bases from frames}

 Let  $Q=Q(\vec M_1,\, ...,\, \vec M_N) $ be as in \eqref{defQ}.  
 Without loss of generality, we can assume   $Q\subset   [-\frac 12 ,\  T-\frac 12)^d $ for some $T>0$. From  \cite[Theorem 2]{KN2}  follows that a basis of $L^2(Q)$   can be extracted from $E ((\mbox{$\frac 1T$} \Z)^d)$, which is an orthogonal basis of 
  $[0,T)^d$  and  an exponential frame of $L^2(Q)$.  When $T$ is an integer, it is easy to verify that  
   $E ((\mbox{$\frac 1T$} \Z)^d)\supset \s ( \mbox{$\vec{\frac 1T}$})$, where $\s(\vec\delta)$  as in 
  \eqref{d2-def-S}, and    $\vec a=(a,\, ...,\, a)$ when  $a\in\R$. Indeed
  \begin{align*} E ((T^{-1} \Z)^d)  & = \{e^{\frac{2\pi i}{T} (n_1  x_1+...+n_dx_d)}\}_{(n_1, ... n_d)\in\Z^d} \\ & =   \bigcup_{j_1,...,j_d=0}^{T-1}  \{e^{ 2\pi i  ((m_1 +\frac{j_1}{T})   x_1+...+(m_d +\frac{j_d}{T})  x_d)}\}_{(m_1, ... m_d)\in\Z^d}\\ &\supset \bigcup_{j =0}^{T-1}   \{e^{ 2\pi i  ((m_1 +\frac{j }{T})   x_1+...+(m_d +\frac{j }{T})  x_d)}\}_{(m_1, ... m_d)\in\Z^d} \\  &=\s(\mbox{$\vec{\frac 1T}$} ).
   \end{align*}
  Let $  \bar T =\sup_{1\leq p\ne q\leq N} ||\vec M_p-\vec M_q||_\infty $ be the smallest positive integer 
 for which $Q\subset [0, \bar T)^d$. If \eqref{e2-cond=delta}  in  Theorem \ref{T1-special-delta} is   satisfied,  then $\s ( \mbox{$ \frac{\vec 1}{ \ov T} $})$
  is an exponential basis of $L^2(Q)$.   For that we need    $  \l \vec M_p-\vec M_q,\,   \vec 1\r \not \in \Z  \bar T$   for every $1\leq p\ne q\leq N$.  Otherwise, we can let  $L>{\bar T}$  be the smallest positive integer  for which  $\frac 1{L}\l \vec M_p-\vec M_q,\,   \vec 1\r \not \in \Z$ and conclude that $\s(\mbox{$\vec{\frac 1L }$})$   is an exponential basis of  $L^2(Q)$  extracted from the exponential frame $E((\mbox{$ {\frac 1L }$}\Z)^d)$.  We have proved the following 
 
\begin{Cor}\label{Cor-extracted1} Let $Q$   be defined as  above. 
   We can find  an integer $L>0$ for which   $Q\subset  [-\frac 12 ,\ L-\frac 12)^d $  and 
   $\s(\mbox{$\vec{\frac 1L}$} ) $ is  a Riesz basis of $L^2(Q) $.  
   \end{Cor}


From   Lemma \ref{lemma3}  and Theorem \ref{Cor-equiv-bases-frames} we have the following.

 \medskip
 \begin{Cor}\label{Cor-complement} Under the assumptions of  Corollary \ref{Cor-extracted1},   
 the set $\s(\mbox{$\vec{\frac 1L}$})$ is a Riesz basis of $L^2(Q)$ if and only if 
 $E((\mbox{$ {\frac 1L}$}\Z)^d )-\s(\mbox{$\vec{\frac 1L}$}) $ is a Riesz basis  on $L^2([-\frac 12 ,\ L-\frac 12)^d-Q)$.
 \end{Cor}

    
  %

 \medskip

 \section {Estimating the frame constants}\label{S-est-eigen}
 
 Let  $Q=Q(\vec M_1,\, ...,\, \vec M_N)$ be as in \eqref{defQ}, and let  $\B=\B(\delta_1,\, ...,\, \delta_N)$ as in \eqref{d2-defofB} be a a Riesz basis on $L^2(Q)$ with optimal frame constants  $0<\lambda\leq \Lambda$. By Theorem   \ref{T1-N-interv} and   \eqref{def-alpha} and \eqref{def-beta},   $\Lambda$ and $\lambda$ are the maximum and minimum eigenvalues of the matrices 
  ${\bf A}= \{\alpha_{i,j}\}_{1\leq i, j\leq N}$ and ${\bf B}= \{\beta_{p,q}\}_{1\leq p,q\leq N}$, with 
  \begin{equation}\label{def-alphabeta} \alpha_{i,j}= \sum_{p=1}^{N}e^{2\pi i \l\vec \delta_i-\vec \delta_j,\ \vec M_p \r},     \qquad
  \beta_{p,q}=  \sum_{j=1}^{N}e^{2\pi i \l\vec M_p-\vec M_q ,\ \vec \delta_j \r}.    \end{equation}
 %
When $\B=\s(\vec \delta)$ is as in \eqref{d2-def-S},   the frame constant of $\B$  are the maximum and minimum eigenvalues of  the matrix  ${\bf\tilde B}=\{\tilde \beta_{p,q}\}_{1\leq p,q\leq N}$ defined in \eqref{def-betatilde-Thm13}.

Gershgorin theorem   provides a powerful tool for estimating the eigenvalues of  complex-valued matrices.   
%
It states that each eigenvalue of a square matrix  ${  M}=\{m_{i,j}\}_{1\leq i,j\leq n}$ is in at least one of the disks
$ D_j=\{z\in\C \, :\, |z-m_{j,j}|\leq R_j\}, 
$ 
and in at least one of the disks 
$ D'_j=\{z\in\C \, :\, |z-m_{j,j}|\leq C_j\}, 
$ 
where  $R_j$ (resp. $C_j$) are the sum of the off-diagonal elements of the $j-$th row (column) of ${\bf M}$, i.e.
\begin{equation}\label{def-RjCj}
R_j=\sum_{i =1\atop{i\ne j}}^n |m_{j, i}|,   \quad C_j=\sum_{i =1\atop{i\ne j}}^n |m_{i,j}| .
\end{equation}
See \cite{G}, and also   \cite[pg. 146]{MM} and \cite{BM}.
 
Observe that if    $|m_{j,j}|> \max\{R_j, C_j\} $ for every $j$, (i.e., if $M$ is {\it diagonally dominant}),   then $M$ is nonsingular. 
  
 The following refinement of Gershgorin theorem  is  in  \cite{AB}.
  \begin{Thm}\label{T-gen-gers}
  Let  ${  M}$ be an Hermitian matrix with eigenvalues $\lambda_1$,...,\,  $ \lambda_n$. Let $R_j=C_j$ be as in \eqref{def-RjCj}. 
 We have
 \begin{equation}\label{eigenv1}
 \max_{1\leq j\leq n}\{m_{j,j}-R_j\}\leq   \max_{1\leq j\leq n}\lambda_{j}\leq \max_{1\leq j\leq n}\{m_{j,j}+R_j\}$$$$ \quad \min_{1\leq j\leq n}\{m_{j,j}-R_j\}\leq \min_{1\leq j\leq n}\lambda_j  \leq \min_{1\leq j\leq n}\{m_{j,j}+R_j\}.
 \end{equation}
\end{Thm}
   %
We can use Theorem \ref{T-gen-gers}  to estimate the optimal  frame constants  of $\B$ and $S(\delta)$.

\begin{Thm}\label{T6-est-eigen} Let $\B$   be a Riesz basis of $L^2(Q)$. 
Let \begin{equation}\label{def-ri} r_i=    \sum_{1\leq j\leq N \atop{j\ne i}}\!\!\left(  1 -\frac 4{N^2}\sum_{p,q=1\atop {p<q}}^N \sin^2( \pi   \l  \vec \delta_i-\vec \delta_j,
 \ \vec M_p-\vec M_q\r) \right)^{\frac 12},
 $$$$
\  \rho_p= \sum_{1\leq q\leq N \atop{q\ne p}}\!\!\left( 1-\frac 4{N^2}\sum_{i,j=1\atop{i<j}}^N \sin^2( \pi   \l  \vec \delta_i-\vec \delta_j,
 \ \vec M_p-\vec M_q\r)\right)^{\frac 12}\!\!\!.
 \end{equation} 
The optimal frame constants of $\B$ satisfy
\begin{equation}\label{e-ineq-gersch}
 N(1-    \min_{1\leq i\leq N\atop{1\leq p\leq N}}\{r_i,\    \rho_p\}\leq\lambda \leq \Lambda \leq  N(1+   \max_{1\leq i\leq N\atop{1\leq p\leq N}} \{r_i,\    \rho_p\}) 
\end{equation}
b) If $\B=\s(\vec \delta)$, we have
\begin{equation}\label{e-2nd-ineq.Ger}
 N(1-  \min_{1\leq p\leq N} s_p)\leq \lambda\leq \Lambda \leq N(1+ \max_{1\leq p\leq N} s_p)
\end{equation}
where $\dsize s_p=\sum_{1\leq q \leq N\atop{q\ne p}} \left|\frac{\sin(\pi N\l \vec \delta, \vec M_q-\vec M_p\r)} {N\sin(\pi  \l \vec \delta, \vec M_q-\vec M_p\r)}\right|$.
 \end{Thm}
 
 \begin{proof} a) The optimal frame constants of $\B$ are the maximum and minimum eigenvalues of the matrices $\A$ and ${\bf B}$. In view of  \eqref{def-alphabeta}, it is easy to verify that   
\begin{align*}\sum_{1\leq j\leq N \atop{j\ne i}}|\alpha_{i, j}|  &= \sum_{1\leq j\leq N \atop{j\ne i}}\left(N +2  \sum_{p,q=1\atop {p<q}}^N \cos(2\pi   \l  \vec \delta_i-\vec \delta_j,
 \ \vec M_p-\vec M_q\r)\right)^{\frac 12}\\ & =
\sum_{1\leq j\leq N \atop{j\ne i}}\left(N    + 2\sum_{p,q=1\atop {p<q}}^N(1-2 \sin^2( \pi   \l  \vec \delta_i-\vec \delta_j,
 \ \vec M_p-\vec M_q\r) \right) ^{\frac 12}
 \\ & =
\sum_{1\leq j\leq N \atop{j\ne i}}\left(N^2   - 4\sum_{p,q=1\atop {p<q}}^N \sin^2( \pi   \l  \vec \delta_i-\vec \delta_j,
 \ \vec M_p-\vec M_q\r) \right)^{\frac 12}
\\ & = 
N   r_i.
 \end{align*}
We can prove that  $\sum_{1\leq q\leq N \atop{q\ne p}}|\beta_{p,q}| =N  \rho_p $  in a similar manner.

 Since $\A$ and ${\bf B}$ have the same eigenvalues, we can apply Theorem \ref{T-gen-gers} with $m_{j,j}=N$ and  $R_i=Nr_i$ or  $R_i=N\rho_i$, and \eqref{e-ineq-gersch} follows.
 
 b) When $\B=\s(\vec \delta)$ we can apply Theorem \ref{T-gen-gers} to the matrix ${\bf \tilde B}$; the inequality 
  \eqref{e-2nd-ineq.Ger} follows from   \eqref{def-betatilde-Thm13}.
 \end{proof}

\begin{Cor}  a) Let $\B$, $\lambda$ and $\Lambda$  be as in  in Theorem \ref{T6-est-eigen}.   
%
If \begin{equation}\label{e-cond-sin}    \min_{1\leq p,q\leq N\atop{p<q}} \sin^2( \pi   \l  \vec \delta_i-\vec \delta_j,
 \ \vec M_p-\vec M_q\r)   \ge \frac{N}{2(N-1)}\left(1-\left(\frac{1-a}{N -1}\right)^2\right)\end{equation}  for some $0<a<1$, we have 
 $$a N\leq \lambda \leq \Lambda \leq (2-a) N.
 $$
 
 b) Let $\B=\s(\vec \delta)$, with  $  \min_{1\leq p,q\leq N\atop{p<q}}|\sin(\pi  \l \vec \delta, \vec M_q-\vec M_p\r)|\ge \frac{ N-1}{N (1-a)}$. Then,
 $$
Na \leq \lambda\leq \Lambda \leq N\left(2+a\right).
 $$
 \end{Cor}
 \begin{proof}
 a)  Let $\dsize s=\min_{1\leq p,q\leq N\atop{p<q}} |\sin ( \pi   \l  \vec \delta_i-\vec \delta_j,
 \ \vec M_p-\vec M_q\r)| $.  By \eqref{def-ri} and \eqref{e-cond-sin},   
  $ 
  r_j\leq   
(N-1)\left(1   - \frac { 2( N-1) } {N }  s^2 \right)^{\frac 12}\leq 1-a.
 $ A similar  inequality holds for    $\rho_p$.
 By Theorem \ref{T6-est-eigen}, part a) of the corollary is proved.
 
 \medskip
 For part b), we let  $\dsize s= \min_{1\leq p,q\leq N\atop{p<q}} |\sin(\pi  \l \vec \delta, \vec M_q-\vec M_p\r)| $. With the notation of   Theorem \ref{T6-est-eigen}  b),
 $$
  s_p\leq \sum_{1\leq q \leq N\atop{q\ne p}} \frac 1 {N|\sin(\pi  \l \vec \delta, \vec M_q-\vec M_p\r)|}\leq  \frac{N-1}{N s}. 
 $$ 
By \eqref{e-2nd-ineq.Ger}, the proof of the corollary is concluded
 \end{proof}


\begin{thebibliography}{999}
 %

 
 \bibitem{AB} Anderson, N.; Best,  G.
{\it A Gerschgorin-Rayleigh inequality for the eigenvalues of Hermitian matrices.}
Linear and Multilinear Algebra 6 (1978/79), no. 3, 219--222.

%

  
\bibitem{BK}   Bezuglaya, L.;    Katsnelson, Y. 
{\it The sampling theorem for functions with limited multi-band spectrum. } 
Z. Anal. Anwendungen 12 (1993), no. 3, 511--534. 

\bibitem{BM}
Brualdi, R.; Mellendorf, S. {\it Regions in the complex plane containing the eigenvalues of a matrix}. Amer. Math. Monthly 101 (1994), no. 10, 975--985. 

\bibitem  {Cr}   Christiansen, O.  
{\it An Introduction to Frames and Riesz Bases }, Birkh\"auser, 2003. 
 

\bibitem{DS}   De Carli, L.; Shaikh Samad, G. {\it One-parameter groups of operators and discrete Hilbert transforms}, to appear on the Canadian Mah. Bullettin, http://dx.doi.org/10.4153/CMB-2016-028-7
 (2016)
 


\bibitem{DK} L. De Carli and A. Kumar, {\it Exponential bases on two dimensional trapezoids},  Proc. Amer. Math. Soc. 143 (2015), no. 7, 2893--2903.



\bibitem{F}  Fuglede, B. {\it Commuting self-adjoint partial differential operators and a group theoretic
problem}, J. Funct. Anal. 16 (1974), 101--121.


\bibitem {G} Gerschgorin, S.  {\it \"Uber die Abgrenzung der Eigenwerte einer Matrix.} Izv. Akad. Nauk. USSR Otd. Fiz.-Mat. Nauk 6, 749--754, 1931.
  %
  \bibitem{GL}  Grepstad, S.;  Lev, N. {\it Multi-tiling and Riesz bases}. Adv. Math. 252 (2014), 1--6.

  %
  \bibitem{GR}  Gunning, R;  Rossi, H. {\it Analytic Functions of Several Complex Variables}, AMS Chelsea Pub., (2009).

\bibitem{Heil}  Heil, C. {\it  A basis theory primer}, Appl.  Num. Harm. Analysis,
 Birkh\"auser  (2011).
 
\bibitem{Ka}  Kadec, M.I.  {\it The exact value of the Paley-Wiener constant}, Sov. Math. doklady, 5, 559-561, (1964).
\bibitem{K}   Kolountzakis, M. {\it The study of translational tiling with Fourier Analysis.}  {\it Fourier Analysis and
Convexity},  131--187. Birkh\"auser, 2004.

 \bibitem{K2}  Kolountzakis, M., {\it Multiple lattice tiles and Riesz bases of exponentials,}
Proc. Amer. Math. Soc. 143 (2015), 741--747.

\bibitem{KN}   Kozma, G.; Nitzan, S.  {\it Combining Riesz bases} Invent. math.
 199,  (2015) no. 1, pp 267--285.
(2012).
\bibitem{KN2}  Kozma, G.; Nitzan, S.  {\it Combining Riesz bases in $\R^n$} arXiv:1501.05257 
(2015)
\bibitem{I} 
  Ingham, A.E. {\it Some trigonometrical inequalities with applications to the theory
of series}, Math. Z. 41 (1936), 367--379.

\bibitem{La}   Laba, I. {\it Fuglede's conjecture for a union of two interval}, Proc. AMS 129 (2001), 2965--2972.

\bibitem{Lae}  Laeng, E. {\it Remarks on the Hilbert transform and some families of multiplier operators related to it}, 
 Collect. Math. 58(1) (2007), 25--44.
 
 \bibitem{Lang}  Lang, S. {\it Linear agebra}, third edition, Springer-Verlag (1987)
 
\bibitem{L} Lev, N. {\it Riesz bases of exponentials on multiband spectra}. Proc.
Amer. Math. Soc. 140 (2012), no. 9, 3127-3132.


\bibitem  {Levinson}  Levinson, N. {\it
On non-harmonic Fourier series}.
Ann. of Math. (2) 37 (1936), no. 4, 919--936.

\bibitem{LRW} Lagarias, J..; Reeds, J.; Wang, Y. {\it Orthonormal bases of exponentials for the n-cube}. Duke Math. J. 103 (2000), no. 1, 25--37. 

\bibitem{LS}   Lyubarskii, Y.;  Seip, K. {\it Sampling and Interpolating
Sequences for Multiband-Limited
Functions and Exponential Bases
on Disconnected Sets}, The Journal of Fourier Analysis and Applications
 3, no. 5, (1997) 598--615.
 %
 \bibitem {Marzo}
  Marzo,  J.{\it Riesz basis of exponentials for a union of cubes in $R^{d}$}
 
 \bibitem{MM}   Marcus, M.;  Mint, H.{\it A Survey of Matrix Theory and Matrix Inequalities},
Prindle, Weber  and Schmidt, Boston, 1964.



\bibitem{NOU}   Nitzan, S.;   Olevskii,  A.; Ulanovskii, A. {\it Exponential frames for unbounded sets}, arXiv:1410.5693 (2014).

\bibitem{PW}   Paley, R.;  Wiener, N. {\it Fourier Transforms in the Complex Domain},  Amer. Math.
Soc. Colloq. Publ., 19, Amer. Math. Soc., New York, 1934 .

\bibitem{P}  Pavlov, B.{\it Basicity of an exponential system and Muckenhoupt's condition}, Soviet Math.
Dokl. 20 (1979)  655--659.

 

\bibitem{S2}   Seip, K.  {\it On the connection between exponential bases and certain related sequences in $L^2(-\pi, \pi)$.} J. Funct. Anal. 130 (1995), no. 1, 131--160. 

\bibitem{SZ}  Sun, W.; Zhou, X. {\it On the stability of multivariate trigonometric systems}. J. Math Anal. Appl. 235 (1999), 159--167.
 
\bibitem{Y} Young, R. M. {\it An Introduction to Nonharmonic Fourier Series}, Academic Press, New York,
1980.
 
 
 \bibitem{W} Y. Wang, Wavelets, tiling, and spectral sets, Duke Math. J. 114 (1) (2002) 43--57.

\end{thebibliography}
\end{document}